\documentclass[leqno,11pt]{article}
\usepackage{amsmath, amssymb}
\usepackage{amsthm}

\numberwithin{equation}{section}

\theoremstyle{plain} %
\newtheorem{theorem}{\noindent\sc \bf{Theorem}}[section] %
\newtheorem{lemma}[theorem]{\noindent\sc \bf{Lemma}}

\newtheorem{proposition}[theorem]{\noindent\sc \bf{Proposition}}

\theoremstyle{definition} %

\newcommand{\C}{\mathbb{C}}
\newcommand{\NN}{\mathbb{N}}
\newcommand{\R}{\mathbb{R}}
\newcommand{\Q}{\mathbb{Q}}
\newcommand{\Z}{\mathbb{Z}}

\newcommand{\SL}{\mathrm{SL}}

\def\ord{\mathop{\rm ord}\nolimits}

\newcommand{\grH}{\mathfrak{H}}

\newcommand{\e}{\mathbf{e}}
\newcommand{\act}[1]{\langle #1  \rangle}
\newcommand{\prm}{_{\mathrm{prm}}}
\newcommand{\tp}{{}^t}
\renewcommand{\Im}{{\rm Im}}
\newcommand{\cD}{\mathcal{D}}

\newcommand{\cF}{\mathcal{F}}

\newcommand{\cH}{\mathcal{S}}

\newcommand{\cM}{\mathcal{M}}

\newcommand{\cS}{\mathcal{S}}

\newcommand{\cV}{\mathcal{V}}

\newcommand{\inv}{^{-1}}

\newcommand{\x}{^{\times}}
\newcommand{\vecm}[1]{\left( \begin{array}{c} #1 \end{array} \right)}
\newcommand{\Nmat}[2]{\left( \begin{array}{cc} #1 \\ #2 \end{array} \right)}

\renewcommand{\pmod}[1]{\; (\bmod \; #1)}
\newcommand{\up}{^{\uparrow}}
\newcommand{\down}{^{\downarrow}}
\newcommand{\oo}{_{\infty}}

\begin{document}

\begin{center}
 {\Large \textbf{A  characterization of the Maass space \\ on
 $O(2,m+2)$  by symmetries\\[0.8cm]}}
{\large \textbf{Bernhard Heim and Atsushi Murase}}
\end{center}

\noindent
In this paper, we define certain symmetries for
automorphic forms on $O(2,m+2)$ and show that the space of automorphic
forms
satisfying these symmetries coincides with the Maass space, the image of
Saito-Kurokawa lifting.

\bigskip


\bigskip

%%%%%%%%%%%%%%%%%%%%%%%%%%%%%%%%%%%%%%%%%%%%%%%%%%%%%%%%%%
% Introduction
%%%%%%%%%%%%%%%%%%%%%%%%%%%%%%%%%%%%%%%%%%%%%%%%%%%%%%%%%%

\section{Introduction}
In \cite{H1}, the first named author introduced 
certain symmetries for Siegel modular forms of even degree
and  showed that the space of Siegel modular forms of degree two
 satisfying the
symmetries coincides with the so-called \emph{Maass
 Spezialschar},
which is the space of Siegel modular forms whose Fourier coefficients
satisfy Maass relations.
Note that this space coincides with the image of Saito-Kurokawa lifting.
Later Bringmann and Heim proved certain symmetries for Jacobi Eisenstein series
of degree two (\cite{BH}). 

On the other hand,
Oda (\cite{O}) and Rallis-Schiffmann (\cite{RS}) independently
studied a theta lifting from elliptic modular forms of integral
or half-integral weight to automorphic forms on the orthogonal
group $O(2,m+2)$.
Note that, if $m=1$, the theta lifting coincides 
with the Saito-Kurokawa lifting.
Later Gritsenko (\cite{G}) and Sugano (\cite{Su})
studied the theta lifting in terms of Jacobi forms of degree $1$.
It is known  that the image of the theta lifting
coincides with the space of holomorphic automorphic forms whose Fourier
coefficients satisfy the Maass relation (\cite{Su}).
Thus
 it is natural to ask whether certain symmetries characterizes the
Maass space in the general orthogonal group case.
The object of the paper is to give an affirmative answer to this question.

In this paper, we introduce symmetries of automorphic forms on
$G=O(2,m+2)$ arising from two embeddings of $\SL_2$ into $G$,
and show that the space of holomorphic automorphic forms on $G$
satisfying the symmetries coincides with the Maass space.

The paper is organized as follows.
In Section 2, we first recall the definitions of automorphic forms
on $G=O(2,m+2)$ and the Maass space.
After defining certain symmetries
for automorphic forms on $G$, 
we state the main result  of the paper (Theorem 2.2):\ 
The space of automorphic
 forms satisfying these symmetries coincides with the space
of those satisfying Maass relations.
As a direct consequence of the characterizaion 
of the Maass space by symmetries, we
show that the restriction mapping induced by an embedding
$G'=O(2,m+1)\hookrightarrow G=O(2,m+2)$ maps the Maass space on $G$ to
that on $G'$.
The proof of Theorem 2.2 is carried out  in Section 3. 
By using some combinatorics,
we prove an algebraic result (Proposition 3.2), from
which Theorem 2.2 follows.

\section*{Notation}
The upper half plane is denoted by $\grH=\{z\in\C\mid\Im(z)>0\}$.
For a real symmetric matrix $R$ of degree $n$, we put
$R(x,y)=\tp x R y$ and $R[x]=\tp x R x$ for $x,y\in \C^n$. 
For a condition $P$, we put 
\[
 \delta(P)=
\begin{cases}
 1&\text{if $P$ holds},\\
 0&\text{otherwise}.
\end{cases}
\]
Denote by $\NN$  the set of natural numbers.
We put $\e[z]=\exp(2\pi\sqrt{-1}z)$ for $z\in\C$.

\section{Main results}
\subsection{The orthogonal group $G$}
Let $S$ be a positive definite even integral symmetric matrix
of degree $m$. 
We put
\[
 Q_1=
\left(
\begin{matrix}
 &&1 \\ &-S& \\ 1 &&
\end{matrix}
\right),
\ 
Q=
\left(
\begin{matrix}
 &&1 \\ &Q_1& \\ 1 &&
\end{matrix}
\right).
\]
In the following, we include the case of $m=0$.
Note that the signatures of $Q_1$ and $Q_2$ are $(1,m+1)$ and $(2,m+2)$,
respectively.

Let
\begin{align*}
 L_0&=\Z^{m}, L_0^*=S\inv L_1, V_0=L_0\otimes_{\Z}\Q=\Q^{m},\\
 L_1&=\Z^{m+2}, L_1^*=Q_1\inv L_1, V_1=L_1\otimes_{\Z}\Q=\Q^{m+2},\\
 L&=\Z^{m+4}, L^*=Q\inv L, V=L\otimes_{\Z}\Q=\Q^{m+4}.
\end{align*}

Let $G=O(Q)$ be the orthogonal group of $Q$ and
 $G\oo^+$  the identity component of $G\oo=G(\R)$. Let
\[
 \cD=\left\{Z=
\vecm{\tau \\ w \\ z}\in \C^{m+2}
\mid \tau,z\in \grH,w\in \C^m,Q_1[\Im(Z)]=\Im (\tau) \Im (z)-\dfrac{1}{2}S[\Im (w)]>0
\right\}.
\]
As is well-known, $\cD$ is a hermitian symmetric domain of type (IV).
We often write $(\tau,w,z)$ for $\vecm{\tau\\w\\z}\in\cD$.
We define an action of $G\oo^+$ on $\cD$ and an automorphic factor
$J\colon G\oo^+\times~ \cD\to \C\x$ by
$g \widetilde Z=\widetilde{g\act{Z}}\, J(g,Z)$
for $g\in G\oo^+$ and $Z\in \cD$, where
\[
 \widetilde Z=
\left(
\begin{matrix}
 -2\inv Q_1[Z]\\ Z \\ 1
\end{matrix}
\right)\in \C^{m+4}.
\]
Let $k$ be an integer  and $F$  a function on $\cD$.
For $g\in G\oo^+$, we define the Petersson slash operator by
$(F|_k g)(Z)=J(g,Z)^{-k}F(g\act{Z})$.

\subsection{Embeddings of $\SL_2$ into $G$}
Let $H=SL_2$. For $h=\Nmat{a&b}{c&d}\in H\oo$ and $z\in \grH$, let
$h\act{z}=(az+b)(cz+d)\inv$ and $j(h,z)=cz+d$
as usual.
We define two embeddings $\iota\up$ and
$\iota\down$
 of $H$ into $G$ by
\begin{align*}
 \iota\up(h)&=
\left(
\begin{matrix}
 a&  &  &-b& \\
  & a&  &  & b\\
  &  &1_m& & \\
-c&  &   & d&\\
  &c &   &  &d
\end{matrix}
\right),
\\
\iota\down(h)&=
\left(
\begin{matrix}
 a& -b &  && \\
-c  & d&  &  & \\
  &  &1_m& & \\
&  &   & a&b\\
  & &   & c &d
\end{matrix}
\right)
\end{align*}
for $h=\Nmat{a&b}{c&d}\in H$, respectively.
It is easily verified that $\iota\up(h)$ and $\iota\down(h)$ commute each
 other and that $\iota\up(H\oo),\iota\down(H\oo)\subset G\oo^+$.
A straightforward calculation shows the following.

\begin{lemma}
 For $h=\Nmat{a&b}{c&d}\in H\oo$ and $Z=(\tau,w,z)\in\cD$, we
 have
\[
  \iota\up(h)\act{Z} = \vecm{h\act{\tau} \\ j(h,\tau)\inv w \\
 z-\dfrac{c}{2j(h,\tau)}S[w]},
 \quad J(\iota\up(h),Z)=j(h,\tau)
\]
and
\[
 \iota\down(h)\act{Z}= \vecm{ 
\tau-\dfrac{c}{2j(h,z)}S[w]\\ j(h,z)\inv w \\
 h\act{z}},
 \quad J(\iota\down(h),Z)=j(h,z).
\]
\end{lemma}

\subsection{Automorphic forms}
Let $\Gamma$ be a discrete subgroup of $G\oo$ commensurable with
 $\Gamma(L)=\{\gamma\in G\oo^+\mid \gamma L=L\}$.
We assume that 
\begin{equation}
 \left(
   \begin{matrix}
    1 & -\tp x Q_1 & -2\inv Q_1[x]\\
    0 & 1_{m+2} & x \\
    0 & 0       & 1
   \end{matrix}
\right), \ 
\iota\up(\gamma), \ \iota\down(\gamma')
\in \Gamma
\qquad (x\in \Z^{m+2}, \gamma,\gamma'\in \SL_2(\Z)).
\end{equation}
Note that $\Gamma(L)$ and $\Gamma^*(L)=\{\gamma\in \Gamma(L)\mid
\gamma l\equiv l\pmod{L}\text{ for any }l\in L^*\}$
satisfy this condition.

For  a positive integer $k$,
let $M_k(\Gamma)$ denote the space of holomorphic functions $F$
on $\cD$ satisfying the following two conditions:

\begin{equation}
 F|_k \gamma =F \ \ \text{for any $\gamma\in \Gamma$}.
\end{equation}
\begin{equation}
\text{If $m=0$, $F$ is holomorphic at any cusp of $\Gamma$.}
\end{equation}

Let
\begin{align}
 \Lambda&=\Z\times L_0^*\times \Z,\\
 \Lambda^+&=\{(a,\alpha,b)\in\Lambda\mid a,b, 2ab-S[\alpha]\ge 0\}.
\end{align}
An automorphic form $F\in M_k(\Gamma)$ admits the Fourier expansion
\[
 F(\tau,w,z)=\sum_{(a,\alpha,b)\in
 \Lambda^+}A_F(a,\alpha,b)\e[az-S(\alpha,w)+b\tau].
\]

We say that $\lambda=(a,\alpha,b)\in \Lambda$ is \emph{primitive} if 
$(n\inv a,n\inv\alpha,n\inv b)\not \in \Lambda$ for 
any $n\in \NN, n>1$.
Denote by $\Lambda\prm$ (respectively $\Lambda^+\prm$)
the set of primitive
elements  of $\Lambda$ (respectively $\Lambda^+$).

\subsection{The Maass space and symmetries}
We now define two subspaces of $M_k(\Gamma)$.

Let $M_k^{\cM}(\Gamma)$ be the space of $F\in M_k(\Gamma)$ satisfying
\begin{equation}
 A_F(la,l\alpha,lb)=\sum_{r|l}r^{k-1}A_F((r\inv l)^2ab,(r\inv l)\alpha,1)
\end{equation}
for any $l\in \NN$ and $(a,\alpha,b)\in \Lambda\prm^+$,
where $r$ runs over the positive divisors of $l$.
Note that
\[
A_F(a,\alpha,b)=\sum_{d\in\Z_{>0},\,d\inv(a,\alpha,b)\in\Lambda}d^{k-1}
A_F\left(\dfrac{ab}{d^2},\dfrac{\alpha}{d},1\right)
\]
for $F\in M_k^{\cM}(\Gamma)$.
When $\Gamma=\Gamma^*(L)$, this space
coincides with  the Maass space  introduced by Maass 
(\cite{Ma}) when $m=1$
and Sugano (\cite{Su}) when $m>1$.

To define symmetries, let
\[
 T_n=\{\xi\in M_2(\Z)\mid \det \xi=n\}=\bigcup_{j}SL_2(\Z)\xi_j
\quad(\text{a disjoint union})
\]
for $n\in \NN$.
Define
\begin{align*}
 F|_k T_n^{\uparrow} & =n^{k/2-1}\sum_j
 F|_k\iota\up(n^{-1/2}\xi_j),\\
 F|_k T_n^{\downarrow} & =n^{k/2-1}\sum_j
 F|_k\iota\down(n^{-1/2}\xi_j)
\end{align*}
for $F\in M_k(\Gamma)$. 
Note that
$F|_k T_n^{\uparrow}$ and $F|_k T_n^{\downarrow}$
are not  in $M_k(\Gamma)$ in general.
We define the space 
$M_k^{\cH}(\Gamma)$
 to be
the space of $F\in M_k(\Gamma)$ satisfying 
$F|_k T_n^{\uparrow}=F|_k T_n^{\downarrow}$
for any $n\in \NN$.
It is easy to see that
$F\in M_k^{\cH}(\Gamma)$ if and only if 
$F|_k T_p^{\uparrow}=F|_k T_p^{\downarrow}$
for any prime number $p$.
Observe that, for a prime number $p$, we have
\begin{align*}
 (F|_k  T_p^{\uparrow})
(\tau,w,z)
&=p^{k-1}F\left(p\tau,\sqrt{p}w, z\right)
+p\inv \sum_{c=0}^{p-1}F\left(p\inv(\tau+c),\sqrt{p}\inv w,z\right)\\
(F|_k T_p^{\downarrow})
(\tau,w,z)
&=p^{k-1}F\left(\tau,\sqrt{p}w,pz\right)
+p\inv \sum_{c=0}^{p-1}F\left(\tau,\sqrt{p}\inv w,p\inv(z+c)\right).
\end{align*}

The main result of the paper is stated as follows.

\begin{theorem}
\label{th:main}The Maass space $M_k^{\cM}(\Gamma)$ coincides
 with 
$M_k^{\cH}(\Gamma)$.
\end{theorem}

\subsection{The compatibility with restrictions}

Let $(L_0',S')$ be a quadratic sub-lattice of $(L_0,S)$.
Let $Q_1', Q', G'$ and $\cD'$ be as in 2.1 corresponding to $S'$.
We assume that the inverse image $\Gamma'$
of $\Gamma$ by the embedding $G'\subset G$ satisfies a condition
similar to
(2.1).
Then the restriction of $F\in M_k(\Gamma)$ to $\cD'$ gives rise to a
linear mapping $j\colon M_k(\Gamma)\to M_k(\Gamma')$.
Since the symmetry is compatible with $j$, we have proved the 
following.

\begin{theorem}
 We have $j(M_k^{\cH}(\Gamma))\subset M_k^{\cH}(\Gamma')$
and hence $j(M_k^{\cM}(\Gamma))\subset M_k^{\cM}(\Gamma')$.
\end{theorem}

%\begin{remark}
% The second statement of the theorem may be proved
%using the compatibility of the theta kernels with the restriction mapping.
%\end{remark}

%%%%%%%%%%%%%%%%%
% 
%%%%%%%%%%%%%%%%%

\section{Proof of Theorem \ref{th:main}}

\subsection{Symmetries and Fourier expansion}

\begin{lemma}
 Let $F\in M_k(\Gamma)$.
Then $F\in M_k^{\cH}(\Gamma)$ if and only if
the following holds for any $(a,b,\alpha)\in \Lambda^+$ and
any prime number $p$:
\begin{align*}
 & p^{k-1}A_F(a,p\inv \alpha,p\inv b)-p^{k-1}A_F(p\inv a,p\inv \alpha,
 b)
  +A_F(a,\alpha,pb)-A_F(pa,\alpha,b)=0.
\end{align*}
 Here we make a convention that $A_F(a,\alpha,b)=0$ if
$(a,\alpha,b)\not \in  \Lambda^+$.
\end{lemma}

\begin{proof}
 We have
\begin{align*}
(F|_k  T_p^{\uparrow})(\tau,\sqrt{p}w,z)
 &=p^{k-1}F(p\tau,pw,z)
  + p\inv \sum_{c=0}^{p-1}F(p\inv(\tau+c),w,z)\\
&=p^{k-1}\sum_{(a,\alpha,b)\in \Lambda^+}A_F(a,p\inv \alpha,p\inv b)
         \e[b\tau+az-S(\alpha,w)]\\
&\quad + \sum_{(a,\alpha,b)\in \Lambda^+}A_F(a,\alpha,p b)
         \e[b\tau+az-S(\alpha,w)]
\end{align*}
and
\begin{align*}
 (F|_k  T_p^{\downarrow})(\tau,\sqrt{p}w,z)
 &=p^{k-1}\sum_{(a,\alpha,b)\in \Lambda^+}A_F(p\inv a,p\inv \alpha,b)
         \e[b\tau+az-S(\alpha,w)]\\
&\quad + \sum_{(a,\alpha,b)\in \Lambda^+}A_F(pa,\alpha,b)
         \e[b\tau+az-S(\alpha,w)],
\end{align*}
from which the lemma immediately follows.
\end{proof}

\subsection{The spaces $\cF^{\cM}$ and $\cF^{\cH}$}
For a function $f$ on $\cV=\Q\times V_0\times \Q$
and $r\in \NN$,
we put
\begin{align*}
M(r)f(a,\alpha,b)&=f(r^2ab,r\alpha,1),\\
N(r)f(a,\alpha,b)&=f(ra,r\alpha,rb)\qquad ((a,\alpha,b)\in\cV).
\end{align*}
It is easy to see that
\begin{align}
  M(r)f(a,\alpha,b)&=M(r)f(ma,\alpha,m\inv b),\\
  M(mr)f(a,\alpha,b)=&M(r)f(m^2a,m\alpha,b)=M(r)f(a,m\alpha,m^2b)
\end{align}
for $r,m\in \NN$.
Let $\cF$ be the space of functions on $\cV$
whose support is contained in $\Lambda$.
For $f\in \cF$, a prime number $p$ and $(a,\alpha,b)\in \cV$,
we set
\[
 I_pf(a,\alpha,b)
=
p^{k-1}f(a,p\inv\alpha,p\inv b)
         -p^{k-1}f(p\inv a,p\inv \alpha,b)+f(a,\alpha,pb)
          -f(pa,\alpha,b).
\]
We define two subspaces of $\cF$ as follows:
\begin{align*}
 \cF^{\cM}&=\{f\in \cF\mid N(l)f(X)=\sum_{r|l}
             r^{k-1}M(r\inv l)f(X)
       \text{ for any } l\in \NN \text{ and
 } X\in \Lambda\prm\},\\
\cF^{\cH}&=\{f\in \cF\mid  I_pf(X)=0 
 \text{ for any prime number } p  \text{ and 
         }X\in \Lambda\}.
\end{align*}

 Let $F\in M_k(\Gamma)$ and consider $A_F$ as an element of $\cF$.
In view of (2.6) and Lemma 3.1, we see that
$F\in M^{\cM}_k(\Gamma)$ if and only if $A_F\in \cF^{\cM}$
and that
$F\in M^{\cS}_k(\Gamma)$ if and only if $A_F\in \cF^{\cS}$.
Thus the  proof of Theorem \ref{th:main} is now reduced to that of the 
following result.

\begin{proposition}We have
 $\cF^{\cM}=\cF^{\cH}$.
\end{proposition}

\subsection{Proof of $\ \cF^{\cM}\subset \cF^{\cH}$} 
In this subsection, we let $p$ be a prime number and $f\in \cF^{\cM}$,
and show that $I_p(X)=0$ for any $X\in\Lambda$.

\begin{lemma}
\label{lem:2-3}
Let $l=p^sn\in\NN$ with $s\ge 0, n\in\NN, p\nmid n$.
 For $X\in \Lambda\prm$, we have
\begin{equation}
\label{eq:2-2}
 N(pl)f(X)-p^{k-1}N(l)f(X)=\sum_{r|n}r^{k-1}M(pr\inv l)f(X).
\end{equation}
\end{lemma}

\begin{proof}
The left-hand side of (\ref{eq:2-2}) is equal to
\begin{align*}
&\sum_{r|p^{s+1}n}M(r\inv p^{s+1}n)f(X)-
    p^{k-1}\sum_{r|p^{s}n}M(r\inv p^{s}n)f(X)\\
&=\sum_{j=0}^{s+1}\sum_{r|n}(p^j r)^{k-1}M(p^{s-j+1}r\inv n)f(X)
   -\sum_{j=0}^{s}\sum_{r|n}(p^j r)^{k-1}M(p^{s-j}r\inv n)f(X)\\
&=\sum_{r|n}r^{k-1}M(p^{s+1}r\inv n)f(X),
\end{align*}
which proves the lemma.
\end{proof}

Let $X\in\Lambda$. Then $X=(lc,l\beta,ld)$ with $l=p^sn\in \NN
\;(s\ge 0, n\in\NN, p\nmid n)$ and $ (c,\beta,d)\in \Lambda\prm$.
To simplify the notation, we write $I$ for $I_pf(X)$.
We have
\begin{align*}
 I&=p^{k-1}N(l)f(c,p\inv \beta,p\inv d)
-N(l)f(pc, \beta, d)\\
&\quad -p^{k-1}N(l)f(p\inv c ,p\inv \beta,d)
   +N(l)f(c, \beta, pd).
\end{align*}

First consider the case where
$\beta\in L_0^*-  pL_0^*$.
Then $(pc,\beta,d), (c,\beta,pd)\in \Lambda\prm$
and 
\[
 M(r\inv l)f(pc,\beta,d)=M(r\inv l)f(c,\beta,pd)
\]
 for $r|l$.
Suppose that $s=0$.
Then we have 
\[
 N(l)f(c,p\inv \beta,p\inv d)=N(l)f(p\inv c,p\inv \beta,d)=0,
\]
since $p\inv l \beta\not\in L_0^*$.
It follows that
\begin{align*}
 I&=-N(l)f(pc,\beta,d)+N(l)f(c,\beta,pd)\\
  &=-\sum_{r|l}r^{k-1}M(r\inv l)f(pc,\beta,d)
   +\sum_{r|l}r^{k-1}M(r\inv l)f(c,\beta,pd)\\
&=0.
\end{align*}
Next suppose that $s>0$. By Lemma \ref{lem:2-3}, we have
\begin{align*}
 I&=p^{k-1}N(p\inv l)f(pc,\beta,d)-N(l)(pc,\beta,d)
    -p^{k-1}N(p\inv l)(c,\beta,pd)+N(l)f(c,\beta,pd)\\
  &=-\sum_{r|n}r^{k-1}M(r\inv l)f(pc,\beta,d)
   +\sum_{r|n}r^{k-1}M(r\inv l)f(c,\beta,pd)\\
&=0.
\end{align*}

Next consider the case where
$\beta\in pL_0^*, p|c$ and $p\nmid d$.
Then $(pc,\beta,d), (p\inv c,p\inv \beta,d)\in \Lambda\prm$
and 
\[
 M(r\inv l)f(pc,\beta,d)=M(pr\inv l)f(p\inv c, p\inv \beta,d).
\]
First suppose that $s=0$.
Then $N(l)f(c,p\inv \beta,p\inv d)=0$ since $p\inv ld\not\in \Z$.
By Lemma \ref{lem:2-3}, we have
\begin{align*}
 I&=-N(l)f(pc,\beta,d)-p^{k-1}N(l)f(p\inv c,p\inv \beta,d)
    + N(pl)f(p\inv c,p\inv \beta,d)\\
&=-\sum_{r|l}r^{k-1}M(r\inv l)f(pc,\beta,d)
  +\sum_{r|l}r^{k-1}M(pr\inv l)f(p\inv c,p\inv \beta,d)\\
&=0.
\end{align*}
If $s>0$, we have
\begin{align*}
 I&=p^{k-1}N(p\inv l)(pc, \beta, d)
-N(l)f(pc,\beta,d)\\
&\quad -p^{k-1}N(l)(p\inv c,p\inv \beta,d)
   +N(pl)f(p\inv c,p\inv \beta,d)\\
&=-\sum_{r|n}r^{k-1}M(r\inv l)f(pc,\beta,d)
 +\sum_{r|n}r^{k-1}M(pr\inv l)f(p\inv c,p\inv \beta,d)\\
&=0
\end{align*}
by Lemma 3.4.

We can show that $I=0$ in the other cases in a similar way.

\subsection{Proof of $\ \cF^{\cH}\subset \cF^{\cM}$}
Let $f\in \cF^{\cH}$.
We will show that
\begin{equation}
 \label{eq:fm}
N(l)f(X)=\sum_{r|l}r^{k-1}M(r\inv l)f(X)
\end{equation}
holds for $l\in \NN$ and $X=(a,\alpha,b)\in \Lambda\prm$ by
induction on $bl$.
If $bl=1$, both sides of (\ref{eq:fm}) are equal to $f(X)$.
Suppose that $bl>1$ and let $p$ be a prime factor of $bl$.
Then we have
\begin{equation*}
\label{eq-fh}
  N(l)f(X)=N(l)f(pa,\alpha,p\inv b)
          +p^{k-1}N(p\inv l)f(a, \alpha, b)
          -p^{k-1}N(l)f(a,p\inv \alpha,p^{-2}b).
\end{equation*}

First consider the case where $p\nmid l$. Then $b$ is divisible by
$p$.
Since $(p\inv l a, p\inv l \alpha,p\inv l b)\not\in\Lambda$,
we have $N(p\inv l)f(a, \alpha, b)=0$.
Suppose that $\alpha\in L_0^*-pL_0^*$.
Then $(pa,\alpha,p\inv b)\in \Lambda\prm$
and $N(l)f(a,p\inv \alpha,p^{-2}b)=0$.
By induction, we have 
\begin{align*}
N(l)f(X)&=N(l)f(pa,\alpha,p\inv b)
=\sum_{r|l}r^{k-1}M(r\inv l)f(pa,\alpha,p\inv b)
=\sum_{r|l}r^{k-1}M(r\inv l)f(a,\alpha,b),
\end{align*}
which proves the claim (\ref{eq:fm}).
Next suppose that $\alpha\in pL_0^*$ and $\ord_p b=1$.
A similar argument as above shows that
\begin{align*}
 N(l)f(X)&=N(l)f(pa,\alpha,p\inv b)
=\sum_{r|l}r^{k-1}M(r\inv l)f(pa,\alpha,p\inv b)
= \sum_{r|l}r^{k-1}M(r\inv l)f(a,\alpha,b).
\end{align*}
Suppose that $\alpha\in pL_0^*$ and $\ord_p b\ge 2$.
Then $p\nmid a$ and $(a,p\inv \alpha,p^{-2}b)\in \Lambda\prm$.
By induction, we have
\begin{align*}
 N(l)f(X)&
=N(pl)f(a,p\inv \alpha,p^{-2}b)-p^{k-1}N(l)f(a,p\inv \alpha,p^{-2}b)\\
&=\sum_{r|l}r^{k-1}M(pr\inv l)f(a,p\inv \alpha,p^{-2}b)
+\sum_{r|l}(pr)^{k-1}M(p(pr)\inv l)f(a,p\inv \alpha,p^{-2}b)\\
&\quad
-p^{k-1}\sum_{r|l}r^{k-1}M(r\inv l)f(a,p\inv \alpha,p^{-2}b)\\
&=\sum_{r|l}r^{k-1}M(pr\inv l)f(a,p\inv \alpha,p^{-2}b)\\
&=\sum_{r|l}r^{k-1}M(r\inv l)f(a,\alpha,b).
\end{align*}

Next consider the case where $p|l$.
We let $l=p^sn$ with $s\ge 1, s\in\NN$ and $p\nmid n$.
First suppose that $p\nmid b$.
Then $(p^2 a,p\alpha,b)\in \Lambda\prm$.
By induction, $N(l)f(X)$ ie equal to
\begin{align*}
 &     N(p\inv l)f(p^2 a,p\alpha,b)+p^{k-1}N(p\inv l)f(a,\alpha,b)
      -p^{k-1}\delta(s\ge 2)N(p^{-2}l)f(p^2a,p\alpha,b)\\
&=\sum_{j=0}^{s-1}\sum_{r|n}(p^{j}r)^{k-1} 
     M(p^{s-j-1}r\inv n)f(p^2a,p\alpha,b)
 +\sum_{j=0}^{s-1}\sum_{r|n}(p^{j+1}r)^{k-1} 
     M(p^{s-j-1}r\inv n)f(a,\alpha,b)\\
&\quad  -\delta(s\ge 2)\sum_{j=0}^{s-2}\sum_{r|n}(p^{j+1}r)^{k-1} 
     M(p^{s-j-2}r\inv n)f(p^2a,p\alpha,b)\\
&=\sum_{j=0}^{s-1}\sum_{r|n}(p^{j}r)^{k-1} 
     M(p^{s-j}r\inv n)f(a,\alpha,b)
+(p^sr)^{k-1}M(r\inv n)f(a,\alpha,b)\\
&=\sum_{j=0}^{s}\sum_{r|n}(p^{j}r)^{k-1} 
     M(p^{s-j}r\inv n)f(a,\alpha,b)\\
&=\sum_{r|l}r^{k-1}M(r\inv l)f(a,\alpha,b).
\end{align*}
Suppose that either $\ord_p b=1$ or ``$\ord_p b\ge 2$ and $\alpha\in
L_0^*-pL_0^*$''
holds.
Then $(pa,\alpha,p\inv b)\in \Lambda\prm$.
By induction, $N(l)f(X)$ is equal to
\begin{align*}
 &N(l)f(pa,\alpha,p\inv b)
          + p^{k-1}N(p\inv l)f(a,\alpha,b)
          -p^{k-1}N(p\inv l)f(pa,\alpha,p\inv b)\\
&=\sum_{r|l}r^{k-1}M(r\inv l)f(pa,\alpha,p\inv b)
  +\sum_{r|p\inv l}(pr)^{k-1}M(p\inv r\inv l)f(a,\alpha,b)\\
&\quad -\sum_{r|p\inv l}(pr)^{k-1}M(p\inv r\inv l)f(pa,\alpha,p\inv b)\\
&=\sum_{r|l}r^{k-1}M(r\inv l)f(a,\alpha,b).
\end{align*}
%Suppose that $\ord_p b\ge 2$ and $\alpha\in L_0^*-pL_0^*$.
%Then $(pa,\alpha,p\inv b)\in \Lambda\prm$.
%A similar argument as above shows that
%%\begin{align*}
%% N(l)f(a,\alpha,b)&=N(l)f(pa,\alpha,p\inv b)
%%          + p^{k-1}N(p\inv l)f(a,\alpha,b)
%%          -p^{k-1}N(p\inv l)f(pa,\alpha,p\inv b)\\
%%&=\sum_{r|l}r^{k-1}M(r\inv l)f(a,\alpha,b).
%%\end{align*}
Finally suppose that $\ord_p b\ge 2$ and $\alpha\in pL_0^*$.
Then $p\nmid a$ and $(a,p\inv \alpha,p^{-2}b)\in \Lambda\prm$.
By induction, $N(l)f(X)$ is equal to
\begin{align*}
 &N(pl)f(a,p\inv \alpha,p^{-2} b)
          + p^{k-1}N(p\inv l)f(a,\alpha,b)
          -p^{k-1}N(l)f(a,p\inv \alpha,p^{-2} b)\\
&=\sum_{r|pl}r^{k-1}M(pr\inv l)f(a,p\inv \alpha,p^{-2} b)
 +\sum_{r|p\inv l}(pr)^{k-1}M(p\inv r\inv l)f(a,\alpha,b)\\
&\quad -\sum_{r|l}(pr)^{k-1}M(r\inv l)f(a,p\inv \alpha,p^{-2} b)\\
&=\sum_{j=0}^{s+1}\sum_{r|n}(p^{j}r)^{k-1} 
     M(p^{s-j+1}r\inv n)f(a,p\inv\alpha,p^{-2}b)
+\sum_{j=0}^{s-1}\sum_{r|n}(p^{j+1}r)^{k-1} 
     M(p^{s-j-1}r\inv n)f(a,\alpha,b)\\
&\quad -\sum_{j=0}^{s}\sum_{r|n}(p^{j+1}r)^{k-1} 
     M(p^{s-j}r\inv n)f(a,p\inv\alpha,p^{-2}b)\\
&=
\sum_{r|n}r^{k-1}M(p^{s+1}r\inv n)f(a,p\inv\alpha,p^{-2}b)
+\sum_{j=0}^{s-1}\sum_{r|n}(p^{j+1}r)^{k-1} 
     M(p^{s-j-1}r\inv n)f(a,\alpha,b)\\
&=\sum_{j=0}^{s}\sum_{r|n}(p^{j}r)^{k-1} 
     M(p^{s-j}r\inv n)f(a,\alpha,b)\\
&=\sum_{r|l}r^{k-1}M(r\inv l)f(a,\alpha,b),
\end{align*}
which completes the proof of (\ref{eq:fm}).

\bigskip

\noindent
\textbf{Acknowledgments}

\noindent
The authors gratefully acknowledges partial support from
Max-Planck-Institut f\"ur Mathematik in Bonn and Mathematisches
Forschungsinstitut
Oberwolfach.
The second named
author was partially supported by Grants-in-Aids from JSPS (20540031).

%\bigskip
%%%%%%%%%%%%%%%%%%%%%

\bigskip

\noindent
Bernhard Heim\\
\noindent
German University of Technology in Oman (GUtech), Department of Applied
Information Technology, PO Box 1816 Athaibah, PC 130, Sultanate of Oman,
Corner
of  Beach and Wadi Athaibah Way\\
e-mail: bernhard.heim@gutech.edu.om

\bigskip

\noindent
Atsushi Murase\\
\noindent
Department of Mathematics, Faculty of Science, 
Kyoto Sangyo University, Motoyama, Kamigamo, 
Kita-ku, Kyoto 603-8555, Japan\\
e-mail: murase@cc.kyoto-su.ac.jp

\end{document}